\documentclass{amsart}
\usepackage{amssymb}
\usepackage{amscd}
\usepackage{verbatim}
\usepackage{epsfig}

\begin{document}
\newcommand\Mand{\ \text{and}\ }
\newcommand\Mor{\ \text{or}\ }
\newcommand\Mfor{\ \text{for}\ }
\newcommand\Real{\mathbb{R}}
\newcommand\RR{\mathbb{R}}
\newcommand\im{\operatorname{Im}}
\newcommand\re{\operatorname{Re}}
\newcommand\sign{\operatorname{sign}}
\newcommand\sphere{\mathbb{S}}
\newcommand\BB{\mathbb{B}}
\newcommand\HH{\mathbb{H}}
\newcommand\dS{\mathrm{dS}}
\newcommand\ZZ{\mathbb{Z}}
\newcommand\codim{\operatorname{codim}}
\newcommand\Sym{\operatorname{Sym}}
\newcommand\End{\operatorname{End}}
\newcommand\Span{\operatorname{span}}
\newcommand\Ran{\operatorname{Ran}}
\newcommand\ep{\epsilon}
\newcommand\Cinf{\cC^\infty}
\newcommand\dCinf{\dot \cC^\infty}
\newcommand\CI{\cC^\infty}
\newcommand\dCI{\dot \cC^\infty}
\newcommand\Cx{\mathbb{C}}
\newcommand\Nat{\mathbb{N}}
\newcommand\dist{\cC^{-\infty}}
\newcommand\ddist{\dot \cC^{-\infty}}
\newcommand\pa{\partial}
\newcommand\Card{\mathrm{Card}}
\renewcommand\Box{{\square}}
\newcommand\Ell{\mathrm{Ell}}
\newcommand\WF{\mathrm{WF}}
\newcommand\WFh{\mathrm{WF}_\semi}
\newcommand\WFb{\mathrm{WF}_\bl}
\newcommand\Vf{\mathcal{V}}
\newcommand\Vb{\mathcal{V}_\bl}
\newcommand\Vz{\mathcal{V}_0}
\newcommand\Hb{H_{\bl}}
\newcommand\Ker{\mathrm{Ker}}
\newcommand\Range{\mathrm{Ran}}
\newcommand\Hom{\mathrm{Hom}}
\newcommand\Id{\mathrm{Id}}
\newcommand\sgn{\operatorname{sgn}}
\newcommand\ff{\mathrm{ff}}
\newcommand\tf{\mathrm{tf}}
\newcommand\esssupp{\operatorname{esssupp}}
\newcommand\supp{\operatorname{supp}}
\newcommand\vol{\mathrm{vol}}
\newcommand\Diff{\mathrm{Diff}}
\newcommand\Diffd{\mathrm{Diff}_{\dagger}}
\newcommand\Diffs{\mathrm{Diff}_{\sharp}}
\newcommand\Diffb{\mathrm{Diff}_\bl}
\newcommand\DiffbI{\mathrm{Diff}_{\bl,I}}
\newcommand\Diffbeven{\mathrm{Diff}_{\bl,\even}}
\newcommand\Diffz{\mathrm{Diff}_0}
\newcommand\Psih{\Psi_{\semi}}
\newcommand\Psihcl{\Psi_{\semi,\cl}}
\newcommand\Psib{\Psi_\bl}
\newcommand\Psibc{\Psi_{\mathrm{bc}}}
\newcommand\TbC{{}^{\bl,\Cx} T}
\newcommand\Tb{{}^{\bl} T}
\newcommand\Sb{{}^{\bl} S}
\newcommand\Lambdab{{}^{\bl} \Lambda}
\newcommand\zT{{}^{0} T}
\newcommand\Tz{{}^{0} T}
\newcommand\zS{{}^{0} S}
\newcommand\dom{\mathcal{D}}
\newcommand\cA{\mathcal{A}}
\newcommand\cB{\mathcal{B}}
\newcommand\cE{\mathcal{E}}
\newcommand\cG{\mathcal{G}}
\newcommand\cH{\mathcal{H}}
\newcommand\cU{\mathcal{U}}
\newcommand\cO{\mathcal{O}}
\newcommand\cF{\mathcal{F}}
\newcommand\cM{\mathcal{M}}
\newcommand\cQ{\mathcal{Q}}
\newcommand\cR{\mathcal{R}}
\newcommand\cI{\mathcal{I}}
\newcommand\cL{\mathcal{L}}
\newcommand\cK{\mathcal{K}}
\newcommand\cC{\mathcal{C}}
\newcommand\cX{\mathcal{X}}
\newcommand\cY{\mathcal{Y}}
\newcommand\cP{\mathcal{P}}
\newcommand\cS{\mathcal{S}}
\newcommand\cZ{\mathcal{Z}}
\newcommand\cW{\mathcal{W}}
\newcommand\Ptil{\tilde P}
\newcommand\ptil{\tilde p}
\newcommand\chit{\tilde \chi}
\newcommand\yt{\tilde y}
\newcommand\zetat{\tilde \zeta}
\newcommand\xit{\tilde \xi}
\newcommand\taut{{\tilde \tau}}
\newcommand\phit{{\tilde \phi}}
\newcommand\mut{{\tilde \mu}}
\newcommand\sigmat{{\tilde \sigma}}
\newcommand\sigmah{\hat\sigma}
\newcommand\zetah{\hat\zeta}
\newcommand\etah{\hat\eta}
\newcommand\loc{\mathrm{loc}}
\newcommand\compl{\mathrm{comp}}
\newcommand\reg{\mathrm{reg}}
\newcommand\GBB{\textsf{GBB}}
\newcommand\GBBsp{\textsf{GBB}\ }
\newcommand\bl{{\mathrm b}}
\newcommand{\sH}{\mathsf{H}}
\newcommand{\cte}{\digamma}
\newcommand\cl{\operatorname{cl}}
\newcommand\hsf{\mathcal{S}}
\newcommand\Div{\operatorname{div}}
\newcommand\hilbert{\mathfrak{X}}

\newcommand\Hh{H_{\semi}}

\newcommand\bM{\bar M}
\newcommand\Xext{X_{-\delta_0}}

\newcommand\xib{{\underline{\xi}}}
\newcommand\etab{{\underline{\eta}}}
\newcommand\zetab{{\underline{\zeta}}}

\newcommand\xibh{{\underline{\hat \xi}}}
\newcommand\etabh{{\underline{\hat \eta}}}
\newcommand\zetabh{{\underline{\hat \zeta}}}

\newcommand\zn{z}
\newcommand\sigman{\sigma}
\newcommand\psit{\tilde\psi}
\newcommand\rhot{{\tilde\rho}}

\newcommand\hM{\hat M}

\newcommand\Op{\operatorname{Op}}
\newcommand\Oph{\operatorname{Op_{\semi}}}

\newcommand\innr{{\mathrm{inner}}}
\newcommand\outr{{\mathrm{outer}}}
\newcommand\full{{\mathrm{full}}}
\newcommand\semi{\hbar}

\newcommand\elliptic{\mathrm{ell}}
\newcommand\diffordgen{k}
\newcommand\difford{2}
\newcommand\diffordm{1}
\newcommand\diffordmpar{1}
\newcommand\even{\mathrm{even}}
\newcommand\dimn{n}
\newcommand\dimnpar{n}
\newcommand\dimnm{n-1}
\newcommand\dimnp{n+1}
\newcommand\dimnppar{(n+1)}
\newcommand\dimnppp{n+3}
\newcommand\dimnppppar{n+3}

\setcounter{secnumdepth}{3}
\newtheorem{lemma}{Lemma}[section]
\newtheorem{prop}[lemma]{Proposition}
\newtheorem{thm}[lemma]{Theorem}
\newtheorem{cor}[lemma]{Corollary}
\newtheorem{result}[lemma]{Result}
\newtheorem*{thm*}{Theorem}
\newtheorem*{prop*}{Proposition}
\newtheorem*{cor*}{Corollary}
\newtheorem*{conj*}{Conjecture}
\numberwithin{equation}{section}
\theoremstyle{remark}
\newtheorem{rem}[lemma]{Remark}
\newtheorem*{rem*}{Remark}
\theoremstyle{definition}
\newtheorem{Def}[lemma]{Definition}
\newtheorem*{Def*}{Definition}

\newcommand{\mar}[1]{{\marginpar{\sffamily{\scriptsize #1}}}}
\newcommand\av[1]{\mar{AV:#1}}

\renewcommand{\theenumi}{\roman{enumi}}
\renewcommand{\labelenumi}{(\theenumi)}

\title[Analytic continuation for differential forms]{Analytic
  continuation and high energy estimates for the resolvent of the Laplacian on forms on
asymptotically hyperbolic spaces}
\author[Andras Vasy]{Andr\'as Vasy}
\address{Department of Mathematics, Stanford University, CA 94305-2125, USA}

\email{andras@math.stanford.edu}


\date{June 23, 2012.}
\thanks{The author gratefully
  acknowledges partial support from the NSF under grant number  
  DMS-1068742.}

\begin{abstract}
We show the analytic continuation of the resolvent of the Laplacian on
asymptotically hyperbolic spaces on differential forms, including high
energy estimates in strips. This is achieved by placing the spectral
family of the Laplacian
within the framework developed, and applied to
scalar problems, by the author recently,
roughly by extending the problem across the boundary
of the compactification of the asymptotically hyperbolic space in a
suitable manner. The main novelty is that the non-scalar nature of the
operator is dealt with by relating it to a problem on an asymptotically
Minkowski space to motivate the choice of the extension across the conformal boundary.
\end{abstract}

\maketitle

\section{Introduction}
Suppose that $(X,g)$ is an $n$-dimensional asymptotically hyperbolic
space with an even metric in the sense of Guillarmou
\cite{Guillarmou:Meromorphic}. That is, $g$ is Riemannian on $X$, $X$
has a compactification $\overline{X}$ with boundary defining function
$x$,
and there is a neighborhood $U=[0,\ep)_x\times \pa X$ of $\pa X$ on
which $g$
is of the warped product form $\frac{dx^2+h}{x^2}$, with $h=h(x,.)$ a
smooth family of symmetric 2-cotensors on $\pa X$ whose Taylor series
at $x=0$ is even, and $h(0,.)$ is positive definite. We refer to
\cite{Guillarmou:Meromorphic} for a more geometric version, and to
Graham and Lee \cite[Section~5]{Graham-Lee:Einstein} for how to put an arbitrary
asymptotically hyperbolic metric, i.e. one for which $x^2g$ is
Riemannian on $\overline{X}$ and $|dx|_{x^2g}=1$ at $x=0$,
into a warped product form. We write
$\overline{X}_{\even}$ for $\overline{X}$ equipped with the even
smooth structure, i.e.\ using coordinate charts $[0,\ep^2)_\mu\times O$, $O$ a
coordinate chart in $\pa X$, in the product decomposition above, where
$\mu=x^2$. (So a $\CI$ function on $\overline X$ is in
$\CI(\overline{X}_{\even})$ if and only if its Taylor series has only
even terms at $x=0$.)

Let $\Delta_k$ denote the Laplacian on
$k$-forms on the complete Riemannian manifold $(X,g)$. Thus,
$\Delta_k$ with domain $\CI_c(X;\Lambda^k X)$ is essentially
self-adjoint, and is indeed non-negative,
so in particular $(\Delta_k-\lambda)^{-1}$ exists for
$\lambda\in\Cx\setminus[0,\infty)$. We show that

\begin{thm}\label{thm:main}
The operators
$$
\delta d(\Delta_k-\sigma^2-(n-2k-1)^2/4)^{-1},
\ d\delta (\Delta_k-\sigma^2-(n-2k+1)^2/4)^{-1}
$$
have a meromorphic continuation from $\im\sigma\gg 1$ to $\Cx$ with
finite rank poles and with non-trapping, resp.\ mildly trapping,
high energy estimates in strips
$|\im\sigma|<C$ if $g$ is a non-trapping, resp.\ mildly trapping,
metric.
\end{thm}

Here recall that $g$ non-trapping means that all geodesics approach
$\pa X$ as the time parameter goes to $\pm\infty$, while mildly
trapping, defined in \cite[Section~2]{Vasy-Dyatlov:Microlocal-Kerr},
is an analytic assumption on a model problem near the trapping
(roughly polynomial bounds for the model resolvent) and the nearby
bicharacteristic flow; we recall this briefly at the end of
Section~\ref{sec:analysis}. Non-trapping high-energy estimates mean that for all
$C_0>0$ and $s$ with $s+3/2>C_0$ there is $C>0$ and $R>0$ such that
\begin{equation}\begin{aligned}\label{eq:high-energy-est}
&\|\delta d(\Delta_k-\sigma^2-(n-2k-1)^2/4)^{-1}\|_{\cL(\cY_{\delta
    d}^{s+1},\cX_{\delta d}^{s})}\leq C|\sigma|,\\
&\|d\delta (\Delta_k-\sigma^2-(n-2k+1)^2/4)^{-1}\|_{\cL(\cY_{d\delta}^{s+1},\cX_{d\delta}^{s})}\leq C|\sigma|,\\
&|\im\sigma|<C_0,\  |\re\sigma|>R,
\end{aligned}\end{equation}
where the norms are on suitable (high-energy) Sobolev spaces, namely
\begin{equation*}\begin{aligned}
&\cX_{\delta
d}^s=x^{-\imath\sigma+(n-2k-1)/2}H^s_{|\sigma|^{-1}}(\overline{X}_{\even};\Lambda^k
\overline{X}_{\even}),\\
&\cY_{\delta
d}^{s+1}=
x^{-\imath\sigma+(n-2k-1)/2+2}H^{s+1}_{|\sigma|^{-1}}(\overline{X}_{\even};\Lambda^k
\overline{X}_{\even}),\\
&\cX_{d \delta}^s=\big\{u\in\dist(X):\ x^{\imath\sigma-(n-2k-3)/2}u\in
H^{s+1}_{|\sigma|^{-1}}(\overline{X}_{\even};\Lambda^k\overline{X}_{\even}),\\
&\qquad\qquad\qquad x^{\imath\sigma-(n-2k-3)/2-2}d\mu\wedge u\in
H^{s+1}_{|\sigma|^{-1}}(\overline{X}_{\even};\Lambda^k\overline{X}_{\even})\big\}\\
&\qquad\subset
x^{-\imath\sigma+(n-2k-3)/2}H^s_{|\sigma|^{-1}}(\overline{X}_{\even};\Lambda^k
\overline{X}_{\even}),\\
&\cY_{d\delta}^{s+1}=
\big\{f\in\dist(X):\ x^{\imath\sigma-(n-2k-3)/2-2}f\in
H^{s+1}_{|\sigma|^{-1}}(\overline{X}_{\even};\Lambda^k\overline{X}_{\even}),\\
&\qquad\qquad\qquad x^{\imath\sigma-(n-2k-3)/2-4}d\mu\wedge f\in
H^{s+1}_{|\sigma|^{-1}}(\overline{X}_{\even};\Lambda^k\overline{X}_{\even})\big\}\\
&\qquad\supset
x^{-\imath\sigma+(n-2k-3)/2+4}H^{s+1}_{|\sigma|^{-1}}(\overline{X}_{\even};\Lambda^k
\overline{X}_{\even}).\\
\end{aligned}\end{equation*}
Here the
power of $|\sigma|$ on the
right hand side of \eqref{eq:high-energy-est} is $1$ rather than $-1$ due to the presence of
$d\delta$ and $\delta d$ on the left hand side which are $|\sigma|^2$
times
second order semiclassical differential operators, as recalled below.
Mildly trapping estimates mean that $|\sigma|$ is replaced by
$|\sigma|^{\varkappa+1}$ for a $\varkappa>0$ arising from the
polynomial models on the trapped model. Notice that as $\re\sigma$ is
assumed sufficiently large, the thresholds $(n-2k\pm 1)^2/4$ are
irrelevant in these estimates. Recall also briefly that on a compact
manifold, possibly with boundary, the semiclassical Sobolev spaces are
$L^2$-based Sobolev spaces in which each derivative is weighted with
$|\sigma|^{-1}$, $|\sigma|\geq 1$. In particular
$|\sigma|^{-2}\delta d$, $|\sigma|^{-2}d \delta$ are second order
semiclassical operators.

Denoting the Hodge star operator on $X$ by $*$, and adding a subscript
to the form spaces to denote the form degree, it is straightforward
to check that $*:\cX_{\delta d,k}^s\to\cX_{d\delta,n-k}^s$ and
$*:\cY_{\delta d,k}^{s+1}\to\cY_{d\delta,n-k}^{s+1}$ are isomorphisms,
so the estimates corresponding to coexact and exact forms indeed match
up. Note that under the mapping $k\mapsto n-k$, the threshold
$(n-2k+1)^2/4$ becomes $(n-2k-1)^2/4$.

We also mention that when one only wants to estimate the operators in
Theorem~\ref{thm:main} away from $\pa X$, one can use semiclassical
elliptic regularity to make the differential order of the domain and
target spaces equal. There is a real loss at $\pa X$ in terms of
standard $\overline{X}_{\even}$-derivatives since the operator which
plays a crucial role in our analysis, on an extended space $\tilde X$, ceases to
be elliptic there.

Noting that
$$
\lambda(\Delta_k-\lambda)^{-1}=-\Id+\delta d(\Delta_k-\lambda)^{-1}+d\delta (\Delta_k-\lambda)^{-1},
$$
and noting that strips
$$
|\im\sqrt{\lambda-(n-2k\pm 1)^2/4}|<C
$$
are
comparable (i.e.\ are contained within each other up to changing $C$
by an arbitrarily small amount) as $|\re\sqrt{\lambda-(n-2k\pm 1)^2/4}|\to +\infty$, we
deduce that

\begin{cor}\label{cor:main}
Let $\Sigma$ be the Riemann surface of the functions
$$
\lambda\mapsto \sqrt{\lambda-(n-2k-1)^2/4},\ \lambda\mapsto \sqrt{\lambda-(n-2k+1)^2/4};
$$
thus $\lambda\mapsto\lambda$ defined on $\Cx\setminus[0,\infty)$
extends to a holomorphic function $\varpi$ on $\Sigma$ (cf.\ \cite[p.~722]{Carron-Pedon:Differential}).

The operator family
$$
\lambda\mapsto (\Delta_k-\lambda)^{-1},
$$
has a meromorphic continuation from $\Cx\setminus[0,\infty)$ to the
Riemann surface $\Sigma$
with
finite rank poles apart from a possible infinite rank pole at the
zeros of $\varpi$ (thus including $\lambda=0$),
and with non-trapping, resp.\ mildly trapping,
high energy estimates in strips
$$
|\im\sqrt{\lambda-(n-2k-1)^2/4}|<C
$$
if $g$ is a non-trapping, resp.\ mildly trapping,
metric.
\end{cor}

An analogous theorem on functions, without high energy estimates, is due to Mazzeo and Melrose
\cite{Mazzeo-Melrose:Meromorphic} and Guillarmou
\cite{Guillarmou:Meromorphic}, using the 0-calculus of Mazzeo and
Melrose. A different proof, with high energy estimates, was provided by the author in
\cite{Vasy-Dyatlov:Microlocal-Kerr} and \cite{Vasy:Microlocal-AH}. Also,
an analogous theorem (without high energy estimates)
for the Dirac operator on a conformally compact
spin manifold using the 0-calculus was proved by Guillarmou, Moroianu
and Park \cite{Guillarmou-Moroianu-Park:Eta}. The $L^2$-Hodge theory
was described by Mazzeo in \cite{Mazzeo:Hodge}, again using the
0-calculus. In the context of actual hyperbolic manifolds, i.e.\
quotients of real hyperbolic space (as well as complex and
quaternionic hyperbolic spaces) the resolvent was constructed by
Carron and Pedon \cite{Carron-Pedon:Differential} using explicit
formulae for exact hyperbolic space; this followed the much earlier
results of Donnelly \cite{Donnelly:Differential} identifying the hyperbolic
Laplacian up to unitary equivalence. In the more general
asymptotically hyperbolic setting Kantor \cite{Kantor:Eleven} has
obtained an analytic continuation (without high energy estimates)
except in middle degree using the
0-calculus and Pedon's explicit results, in part based on some notes
provided by Guillarmou on the model case.

This theorem is proved by `conjugating', or more precisely
appropriately modifying, the Laplacian on differential
forms to an operator which has a continuation across the boundary, as
was done in the scalar setting by the author in
\cite{Vasy-Dyatlov:Microlocal-Kerr} and
\cite{Vasy:Microlocal-AH}. However, here we emphasize an `ambient
space' point of view, which, while by no means necessary, is very
enlightening; it uses a one higher dimensional (Minkowski type)
Lorentzian manifold to perform this continuation across the boundary.
Ambient metric constructions in conformal geometry (relating the
`bulk' and the asymptotically hyperbolic boundary) were
introduced by Fefferman and Graham \cite{Fefferman-Graham:Conformal},
see \cite{Graham-Hirachi:Inhomogeneous} for a recent treatment, but
there Ricci flatness was an important consideration, while here this
plays no role, instead merely the $\RR^+$-equivariance is relevant. We
also refer to the recent monograph by Fefferman and Graham \cite{Fefferman-Graham:Ambient} for a more
thorough treatment, including what they call `pre-ambient metrics'
(without a Ricci condition). There is also the very recent work of
Gover, Latini and Waldron \cite{Gover-Latini-Waldron:Poincare} using
the tractor calculus to analyze the geometric connection between asymptotically
hyperbolic and ambient frameworks on differential forms.

The operator obtained in this extension process is an operator $P_\sigma$ acting
on two copies of the form bundle. In order to explain how this arises,
and to motivate the subsequent constructions, we start by considering the d'Alembertian $\Box_{\tilde g}$ on Minkowski space $(\RR^{n+1},\tilde
g)$ and the hyperbolic Laplacian in the next section,
and then finally extending the results to
general $X$ in the Section~\ref{sec:conformal}. The analytic background
is recalled in Section~\ref{sec:analysis}. This is merely a summary of
the relevant parts of \cite{Vasy:Microlocal-AH} and
\cite{Vasy-Dyatlov:Microlocal-Kerr} since no new analytic tools are
required; the set-up in these papers was such that it included
non-scalar operators with scalar principal symbols, which the
Laplacian on forms possesses.

There are no infinite rank poles at the thresholds on functions or top
forms; from the perspective of the present paper this is so since one can
work with a line bundle, i.e.\ by restriction of the form degree one
of the two copies in the sum discussed above becomes trivial. One
should be able to perform a more detailed analysis at the thresholds
to rule out the infinite rank poles in certain other degrees; they are
well-known to occur in middle degree even on hyperbolic space, see
\cite{Donnelly:Differential}. We briefly point out an approach to this
more detailed analysis at the end of Section~\ref{sec:Minkowski}.

While we use $P_\sigma$ and complex absorption
to analyze the asymptotically hyperbolic resolvent, in fact
when combined with analysis of the Klein-Gordon operator
on asymptotically de Sitter spaces, the complex absorption can be
dropped and the
argument is fully reversible. In particular, on functions, this
reversibility
holds in the sense that the poles of the resolvent
of the Laplacian on $X$
correspond (understood in pairs, at $\sigma$ and at $-\sigma$, as a
dual problem also enters),
apart from some integer coincidences, to poles of
$P_\sigma^{-1}$.
This, including the connection
of the Poisson operators and scattering matrices will be discussed in
a companion paper \cite{Vasy:AH-DS-combine}. A concrete application
without complex absorption is the analysis of
\cite{Baskin-Vasy-Wunsch:Radiation} in asymptotically Minkowski spaces
(again, on functions).

The author is grateful to Robin Graham and Colin Guillarmou for
providing some of the references.

\section{Minkowski space, hyperbolic space and de Sitter
  space}\label{sec:Minkowski}
In this section we connect the analysis on the form bundles on Minkowski, hyperbolic and
de Sitter spaces. Here we underemphasize de Sitter space, but in fact
the analysis of the wave operator on forms on it is completely
parallel to our treatment of hyperbolic space, as we point this out
occasionally in what follows.
This connection has a direct extension, with simple
modifications, to the general asymptotically hyperbolic/de Sitter
setting, thus while the present section is a model case, it is the
heart of the paper.

The starting point of analysis is the manifold $\RR^{n+1}$, or rather
$\RR^{n+1}\setminus o$, which is equipped with an $\RR^+$-action given
by dilations: $(\lambda,z)\mapsto \lambda z$.
A transversal to this action is, as a differentiable manifold,
$\sphere^n$, which may
be considered
as the unit sphere with respect to the Euclidean metric, though the
metric properties are not important here (since we are interested in
the Minkowski metric after all).
Thus, writing $(z_1,\ldots,z_{n+1})$ as
the coordinates, let
$$
dz_1^2+\ldots+dz_n^2+dz_{n+1}^2,
$$
be the {\em Euclidean} metric, and
 let $\rho$ be the Euclidean distance function on
$\RR^{n+1}$ from the origin,
namely
$$
\rho=(z_1^2+\ldots+z_n^2+z_{n+1}^2)^{1/2}.
$$
Then $\sphere^n$ is the $1$-level set of $\rho$.
One can identify $\RR^{n+1}\setminus\{0\}$ via the Euclidean polar
coordinate map with
$\RR^+_\rho\times\sphere^n$, namely the map is
$\RR^+_\rho\times\sphere^n \ni(\rho,\omega)\mapsto \rho \omega\in \RR^{n+1}\setminus\{0\}$.

The Minkowski metric is given by
$$
\tilde g=dz_{n+1}^2-(dz_1^2+\ldots+dz_n^2),
$$
and we also consider the Minkowski distance function $r$.
Thus, away from the light cone, where $z_{n+1}^2=z_1^2+\ldots+z_n^2$,
let
$$
r=|z_{n+1}^2-(z_1^2+\ldots+z_n^2)|^{1/2}.
$$
We are interested in $\Box_{\tilde g}$ on differential forms.
To analyze this,
we conjugate $\rho^2\Box_{\tilde g}$ by the Mellin transform $\cM_\rho$
on $\RR^+_\rho\times\sphere^n$, identified with
$\RR^{n+1}\setminus\{0\}$ as above. To be precise, we identify the form bundle
on $\RR^{n+1}\setminus\{0\}$ with the pullback of
$\Lambda\sphere^n\oplus \Lambda\sphere^n$ by decomposing a
differential form into tangential and normal parts {\em relative to the
Euclidean metric}  i.e.\ writing forms as conormal forms plus
orthogonal to these forms, which we think of as tangential
forms. Notice that $T_{\rho}\RR^+\oplus T_\omega\sphere^n$ is an orthogonal
decomposition of $T_{\rho \omega}\RR^{n+1}$ relative to the Euclidean
metric.
Thus, a $k$-form on $\RR^+\times\sphere^n$ is written as
$$
u=u_T+\frac{d\rho}{\rho}\wedge u_N,
$$
where $u_T$ and $u_N$ are respectively $k$ and $k-1$ forms on
$\sphere^n$, depending on $\rho$;
we used $\frac{d\rho}{\rho}$ instead of $d\rho$ due to
homogeneity reasons.
The so-obtained operator,
$$
{P}_{0,\sigmat}=\cM_\rho \rho^{2}\Box_{\tilde g}\cM_\rho^{-1}\in\Diff^2(\sphere^n;\Lambda\sphere^n\oplus \Lambda\sphere^n),
$$
with $\sigmat$ the Mellin dual parameter,
fits into the framework of \cite{Vasy-Dyatlov:Microlocal-Kerr} and
\cite{Vasy:Microlocal-AH}. As an aside, we remark that it will be convenient to shift the
Mellin parameter, or equivalently conjugate $\Box_{\tilde g}$ by a
power of $\rho$; we shall do so later in
\eqref{eq:form-weight-conjugation}, and this is the reason for adding
the cumbersome subscript $0$ to ${P}_{0,\sigmat}$ presently. We
explain the fit in more detail in the general asymptotically hyperbolic
setting in Section~\ref{sec:conformal}, but we briefly indicate why
this happens in terms of the scalar problem using special properties
of the Minkowski metric here.
Thus, the reason for the aforementioned fit into the framework is simple: in the case of the
scalar d'Alembertian on Minkowski space
this was shown in \cite{Vasy-Dyatlov:Microlocal-Kerr}; the
d'Alembertian on forms is scalar on Minkowski space with respect to
the decomposition of the form bundle relative to any basis of
$\RR^{n+1}$, identified with $T_z\RR^{n+1}$ for all $z\in\RR^{n+1}$,
so with respect to this decomposition of the bundle (identifying the
form bundle as a trivial bundle over $\sphere^n$), the
(component-wise) Mellin transform fits into the framework as
claimed. Now, the transition to the tangent plus normal form bundle
decomposition amounts to a conjugation by a bundle endomorphism (we
perform a similar one below) on $\sphere^n$; such a conjugation
preserves all the properties required for the analysis, except
causing a form-degree dependent shift in the subprincipal term
due to the different homogeneities of the forms (degree $k$ on
$k$-forms relative to the above trivialization, vs.\ degree $0$
relative to the tangential plus normal decomposition).

While so far we explained why the Minkowski wave operator on forms can
be analyzed by means of \cite{Vasy-Dyatlov:Microlocal-Kerr} and
\cite{Vasy:Microlocal-AH}, we still need to connect this to
asymptotically hyperbolic and de Sitter spaces.
But in the region in $\sphere^n$ corresponding to the interior of the future
light cone, which can be identified with the hyperboloid
$$
\HH^n:\ z_{n+1}^2-(z_1^2+\ldots+z_n^2)=1,\ z_{n+1}>0,
$$
via the $\RR^+$-quotient,
one can also consider the Mellin transform of $r^2\Box_{\tilde g}$
with respect to the decomposition $\RR^+_r\times\HH^n$, and the
corresponding tangential-normal decomposition of the form bundle {\em
  relative to the Minkowski metric}, to
get
$$
{\tilde P}_\sigmat=\cM_r r^2\Box_{\tilde g}\cM_r^{-1}\in\Diff^2(\HH^n;\Lambda\HH^n\oplus\Lambda\HH^n).
$$
Now, ${\tilde P}_\sigmat$ is not well-behaved at the boundary of the
future light cone, but it is closely related to ${P}_\sigmat$. If
we use coordinates
$$
\omega_j=\frac{z_j}{z_{n+1}},\ j=1,\ldots,n,
$$
on the sphere away from the equator $z_{n+1}=0$ (note that $\omega_j$
is not the $j$th component of $\omega$ with $\sphere^n$ considered as
a subset of $\RR^{n+1}$!), then, with $|.|$ the
Euclidean norm on $\RR^n$,
$$
r=F(\omega)\rho,\ F(\omega)=\sqrt{\frac{1-|\omega|^2}{1+|\omega|^2}}.
$$
Note that
$$
\mu=F^2
$$
is a smooth function on $\sphere^n$ near (its
intersection with) the light cone which vanishes non-degenerately at
the light cone. On the other hand, the Poincar\'e ball model
$\overline{\HH^n}$ of
$\HH^n$ arises by regarding it as a graph over $\RR^n$ in
$\RR^n\times\RR$, and compactifying $\RR^n$ radially (or geodesically)
to a ball, with boundary defining
function, say, $(z_1^2+\ldots+z_n^2)^{-1/2}$, or, $\rho^{-1}$ -- these
two differ by a smooth positive multiple on $\overline{\HH^n}$. As
$r=1$ on $\HH^n$, this means that $F$ is a valid boundary defining
function in the Poincar\'e model, in contrast with the natural $F^2$
defining function of the light cone. In particular, with $y_j$,
$j=1,\ldots,n-1$, denoting local coordinates on $\sphere^{n-1}$, identified
with $\pa \overline{\HH^n}$, hence the light cone within $\sphere^{n-1}$,
differential forms on $\sphere^n$ have the form
$$
c_I d{y}^I+c_J dF^2\wedge d{y}^J,
$$
with $c_I$ and $c_J$ smooth. We remark that pulling back the Minkowski
metric to $\HH^n$, which by definition yields the hyperbolic metric,
a straightforward calculation yields that that
\begin{equation}\label{eq:hyp-metric-form-F}
g=\frac{(dF)^2}{F^2(1-F^2)}+\frac{1-F^2}{2F^2}h({y},d{y}),
\end{equation}
with $h$ the round metric on the sphere $\sphere^{n-1}$; this satisfies $F^2g$ a
smooth metric up to the boundary, $F=0$ (with a polar coordinate
singularity at $F=1$; $F$ and ${y}$ are not valid coordinates
there, though $F$ is still $\CI$ near $F=1$, and the metric is still
$\CI$ there as
well, as can be seen by using valid coordinates), with the coefficients even functions of $F$. The metric $g$ can be put in
the normal form $g=\frac{dx^2+h}{x^2}$ by letting
$x=\frac{F}{1+\sqrt{1-F^2}}$, which is an equivalent boundary defining
function, but this is not necessary here.

We remark at this point that de Sitter space can be approached in a
completely parallel manner.
Namely, in the region in $\sphere^n$ corresponding to the `equatorial
belt', i.e.\ the exterior of the future and past light cones,
which can be identified with the hyperboloid
$$
\dS^n:\ z_{n+1}^2-(z_1^2+\ldots+z_n^2)=-1,
$$
via the $\RR^+$-quotient,
one can also consider the Mellin transform of $r^2\Box_{\tilde g}$
with respect to the decomposition $\RR^+_r\times\dS^n$, and the
corresponding tangential-normal decomposition of the form bundle
  relative to the Minkowski metric, to
get
$$
{\tilde P}_{\dS,\sigmat}=\cM_r r^2\Box_{\tilde g}\cM_r^{-1}\in\Diff^2(\dS^n;\Lambda\dS^n\oplus\Lambda\dS^n).
$$

Returning to $\HH^n$ and
rewriting a form in tangential-normal decomposition with respect to
the Minkowski metric as such with respect to the Euclidean metric,
where $\HH^n$ is identified as an open subset of $\sphere^n$, one has
$$
v_T+\frac{dr}{r}\wedge v_N=\Big(v_T+\frac{dF}{F}\wedge
v_N\Big)+\frac{d\rho}{\rho}\wedge v_N=u_T+\frac{d\rho}{\rho}\wedge u_N,
$$
with
$$
\begin{bmatrix} v_T\\v_N\end{bmatrix}=J \begin{bmatrix}
  u_T\\u_N\end{bmatrix},\qquad
J=\begin{bmatrix}\Id&\frac{dF}{F}\wedge\\0&\Id\end{bmatrix},\qquad
 J^{-1}=\begin{bmatrix}\Id&-\frac{dF}{F}\wedge\\0&\Id\end{bmatrix}.
$$
Since for $f$ taking values in a bundle over $\sphere^n$
$$
\cM_\rho f(\sigmat,\omega)=\int_0^\infty \rho^{-\imath\sigmat} f\,\frac{d\rho}{\rho},
$$
with a similar formula for $\cM_r$, we have, if we identify $\HH^n$
with an open subset of $\sphere^n$ (the interior of the future light
cone), and correspondingly identify the form bundles, on $\CI(\HH^n)$,
\begin{equation}\label{eq:relate-Mellin}
\cM_\rho\rho^2\Box_{\tilde
 g}\cM_\rho^{-1}(\sigmat)=J^{-1}F^{\imath\sigmat-2}\cM_r r^2\Box_{\tilde g}\cM_r^{-1}F^{-\imath\sigmat}J.
\end{equation}

We next compute $\cM_r r^2\Box_{\tilde g}\cM_r^{-1}$; this is feasible
since $\RR^+\times\HH^n$ is an orthogonal decomposition relative to
$\tilde g$. Concretely, the Minkowski metric is
$$
\tilde g=dr^2+r^2 g,
$$
where $g$ is the hyperbolic metric, since by definition the hyperbolic
metric {\em is} the restriction of the Minkowski metric to the
hyperboloid $\HH^n$. This is a conic metric, whose Laplacian was
computed by Cheeger \cite[Equation~(3.8)]{Cheeger:Spectral}.
This is best done relative to a tangential-normal
decomposition of the form bundle of $\RR^{n+1}$ relative to $\HH^n$
and the Minkowski metric, i.e.\ writing forms as conormal forms plus
orthogonal to these forms, which we again think of as tangential
forms. Concretely, following Cheeger's decomposition,
a $k$-form on $\RR^+\times\HH^n$ is written as
$$
v=\tilde v_T+dr\wedge \tilde v_N,
$$
where $v_T$ and $v_N$ are respectively $k$ and $k-1$ forms on $\HH^n$.
Then, in this decomposition, writing $v=(v_T,v_N)$, writing $X=\HH^n$,
$$
\Box_{\tilde g}=\begin{bmatrix}-r^{-2}\Delta_X-r^{2k-n}\pa_r r^{n-2k}\pa_r &
 -2r^{-1}d_X\\ 
2r^{-3}\delta_X & -r^{-2}\Delta_X-\pa_r r^{2(k-1)-n}\pa_r r^{n-2(k-1)}\end{bmatrix},
$$ 
similarly to Cheeger's case with some sign changes due to the
Lorentzian signature of $\tilde g$.
Rewriting in a form that is more useful for homogeneity reasons,
$$
v=v_T+\frac{dr}{r}\wedge v_N,\ v_T=\tilde v_T,\ v_N=r\tilde v_N,
$$
$$
r^2\Box_{\tilde g}=\begin{bmatrix}-\Delta_X-r^{2k-n+2}\pa_r r^{n-2k}\pa_r &
 -2d_X\\
2\delta_X & -\Delta_X-r^3\pa_r r^{2(k-1)-n}\pa_r r^{n-2(k-1)-1}\end{bmatrix}.
$$
Thus, as
\begin{equation*}\begin{aligned}
&r^{2k-n+2}\pa_r r^{n-2k}\pa_r=(r\pa_r)^2+(n-2k-1)r\pa_r\\
&\qquad\qquad=\Big(r\pa_r+\frac{n-2k-1}{2}\Big)^2-\Big(\frac{n-2k-1}{2}\Big)^2,\\
&r^3\pa_r r^{2(k-1)-n}\pa_r
r^{n-2(k-1)-1}=(r\pa_r-2)(r\pa_r+(n-2k+1))\\
&\qquad\qquad=\Big(r\pa_r+\frac{n-2k-1}{2}\Big)^2-\Big(\frac{n-2k+3}{2}\Big)^2,
\end{aligned}\end{equation*}
in this basis we have
\begin{equation}\begin{aligned}\label{eq:hyp-MT-form}
&\Diff^2(X;\Lambda X\oplus\Lambda X)\ni\cM_r r^2\Box_{\tilde g}\cM_r^{-1}\\
&=\begin{bmatrix}-\Delta_X+\Big(\sigmat-\imath\frac{n-2k-1}{2}\Big)^2+\Big(\frac{n-2k-1}{2}\Big)^2 &
-2d_X\\
2\delta_X & -\Delta_X+\Big(\sigmat-\imath \frac{n-2k-1}{2}\Big)^2+\Big(\frac{n-2k+3}{2}\Big)^2\end{bmatrix}.
\end{aligned}\end{equation}
In view of this formula, it is convenient to introduce
$$
\sigma=\sigmat-\imath \frac{n-2k-1}{2}
$$
to simplify some expressions; so shifting the Mellin parameter amounts
to conjugation of $\Box_{\tilde g}$ by $r^{-(n-2k-1)/2}$, i.e.\
considering
\begin{equation}\label{eq:form-weight-conjugation}
\cM_r r^2 r^{(n-2k-1)/2}\Box_{\tilde
  g}r^{-(n-2k-1)/2}\cM_r^{-1}(\sigma)=\cM_r r^2\Box_{\tilde
  g}\cM_r^{-1}(\sigmat),\ \sigmat=\sigma+\imath \frac{n-2k-1}{2}.
\end{equation}
Thus,
\begin{equation}\begin{aligned}\label{eq:hyp-MT-form-mod}
&\cM_r r^2 r^{(n-2k-1)/2}\Box_{\tilde g}r^{-(n-2k-1)/2}\cM_r^{-1}\\
&=\begin{bmatrix}-\Delta_X+\sigma^2+\Big(\frac{n-2k-1}{2}\Big)^2 &
-2d_X\\
2\delta_X & -\Delta_X+\sigma^2+\Big(\frac{n-2k+3}{2}\Big)^2\end{bmatrix}.
\end{aligned}\end{equation}
Combining with \eqref{eq:relate-Mellin} we deduce the following lemma:

\begin{lemma}\label{lemma:global-to-local}
Let
$$
{P}_\sigma=\cM_\rho \rho^2 \rho^{(n-2k-1)/2}\Box_{\tilde g}\rho^{-(n-2k-1)/2}\cM_\rho^{-1}.
$$
Then
\begin{equation*}\begin{aligned}
&\begin{bmatrix}-\Delta_X+\sigma^2+\Big(\frac{n-2k-1}{2}\Big)^2 &
-2d_X\\
2\delta_X & -\Delta_X+\sigma^2+\Big(\frac{n-2k+3}{2}\Big)^2\end{bmatrix}\\
&\qquad\qquad\qquad=JF^{-\imath\sigma+(n-2k-1)/2+2}{P}_\sigma F^{\imath\sigma-(n-2k-1)/2}J^{-1}.
\end{aligned}\end{equation*}
\end{lemma}

While \eqref{eq:hyp-MT-form-mod} is not a diagonal matrix, the off-diagonal terms have a
special structure. In particular, for $v_T$ coclosed and $v_N$ closed
we have that 
\begin{equation*}\begin{aligned}
&\cM_r r^2 r^{(n-2k-1)/2}\Box_{\tilde g}r^{-(n-2k-1)/2}\cM_r^{-1}\begin{bmatrix}v_T\\v_N\end{bmatrix}\\
&=\begin{bmatrix}-\Delta_X+\sigma^2+\Big(\frac{n-2k-1}{2}\Big)^2 &
0\\
0 & -\Delta_X+\sigma^2+\Big(\frac{n-2k+3}{2}\Big)^2\end{bmatrix}\begin{bmatrix}v_T\\v_N\end{bmatrix},
\end{aligned}\end{equation*}
so
\begin{equation*}\begin{aligned}
&\cM_r r^2 r^{(n-2k-1)/2}\Box_{\tilde
  g}r^{-(n-2k-1)/2}\cM_r^{-1}\begin{bmatrix}\delta_X
  d_X&0\\0&d_X\delta_X\end{bmatrix}\\
&=\begin{bmatrix}\delta_X d_X&0\\0&d_X\delta_X\end{bmatrix}\cM_r r^2 r^{(n-2k-1)/2}\Box_{\tilde g}r^{-(n-2k-1)/2}\cM_r^{-1}\\
&=\begin{bmatrix}\Big(-\Delta_X+\sigma^2+\Big(\frac{n-2k-1}{2}\Big)^2\Big) \delta_X d_X&
0\\
0 & \Big(-\Delta_X+\sigma^2+\Big(\frac{n-2k+3}{2}\Big)^2\Big) d_X\delta_X\end{bmatrix}.
\end{aligned}\end{equation*}
Correspondingly, let $\iota_{X,T,k}$, resp.\ $\iota_{X,N,k-1}$ denote the inclusion
maps $\Lambda^kX\to\Lambda^kX\oplus\Lambda^{k-1}X$, resp.\
$\Lambda^{k-1}X\to\Lambda^kX\oplus\Lambda^{k-1}X$ as $0$ in the other
summand, and $\pi_{X,T,k}$, resp.\ $\pi_{X,N,k-1}$ be the projection
maps.
Then, using \eqref{eq:relate-Mellin},
\begin{equation*}\begin{aligned}
&\delta_X
d_X\Big(-\Delta_X+\sigma^2+\Big(\frac{n-2k-1}{2}\Big)^2\Big)\\
&\qquad\qquad=\delta_X
d_X\pi_{X,T,k} \cM_r r^2 r^{(n-2k-1)/2}\Box_{\tilde g}r^{-(n-2k-1)/2}\cM_r^{-1}\iota_{X,T,k},\\
&d_X
\delta_X\Big(-\Delta_X+\sigma^2+\Big(\frac{n-2k+3}{2}\Big)^2\Big)\\
&\qquad\qquad=d_X
\delta_X\pi_{X,N,k-1} \cM_r r^2 r^{(n-2k-1)/2}\Box_{\tilde g}r^{-(n-2k-1)/2}\cM_r^{-1}\iota_{X,N,k-1}.
\end{aligned}\end{equation*}
We then have, via regarding $X=\HH^n$ as a subset of $\tilde X=\sphere^n$
and using the corresponding identification of the form bundles
\begin{equation*}\begin{aligned}
&\delta_X
d_X\Big(-\Delta_X+\sigma^2+\Big(\frac{n-2k-1}{2}\Big)^2\Big)\\
&\qquad=\delta_X
d_X\pi_{X,T,k} JF^{-\imath\sigma+(n-2k-1)/2+2}{P}_\sigma F^{\imath\sigma-(n-2k-1)/2}J^{-1}\iota_{X,T,k},\\
&d_X
\delta_X\Big(-\Delta_X+\sigma^2+\Big(\frac{n-2k+3}{2}\Big)^2\Big)\\
&\qquad=d_X
\delta_X\pi_{X,N,k-1} JF^{-\imath\sigma+(n-2k-1)/2+2}{P}_\sigma F^{\imath\sigma-(n-2k-1)/2}J^{-1}\iota_{X,N,k-1}.
\end{aligned}\end{equation*}
Correspondingly,
for an appropriately chosen inverse
${P}_\sigma^{-1}$ of the
$\cM_\rho$-conjugated operator we have for $\im\sigma\gg 1$,
with $r_X$ denoting restriction to $X$, $e_X$ an extension map from
$X$ to $\tilde X$, i.e.\ $e_X:\CI_c(X;\Lambda X\oplus\Lambda X)\to\CI(\tilde
X;\Lambda\tilde X\oplus\Lambda\tilde X)$ with $r_X e_X=\Id$ that 
\begin{equation}\begin{aligned}\label{eq:form-res-connection}
&\delta_X
d_X\Big(-\Delta_X+\sigma^2+\Big(\frac{n-2k-1}{2}\Big)^2\Big)^{-1}
=\Big(-\Delta_X+\sigma^2+\Big(\frac{n-2k-1}{2}\Big)^2\Big)^{-1}\delta_X
d_X\\
&\qquad=\delta_X
d_X\pi_{X,T,k}
JF^{-\imath\sigma+(n-2k-1)/2}r_X{P}_\sigma^{-1}e_X F^{\imath\sigma-(n-2k-1)/2
  -2}J^{-1}\iota_{X,T,k}\\
\end{aligned}\end{equation}
and
\begin{equation}\begin{aligned}\label{eq:form-res-connection-2}
&d_X
\delta_X\Big(-\Delta_X+\sigma^2+\Big(\frac{n-2k+3}{2}\Big)^2\Big)^{-1}
=\Big(-\Delta_X+\sigma^2+\Big(\frac{n-2k+3}{2}\Big)^2\Big)^{-1}d_X
\delta_X\\
&\qquad=d_X
\delta_X\pi_{X,N,k-1}
JF^{-\imath\sigma+(n-2k-1)/2}r_X{P}_\sigma^{-1}e_X F^{\imath\sigma-(n-2k-1)/2-2}J^{-1}\iota_{X,N,k-1}.\\
\end{aligned}\end{equation}
Concretely, we have the following lemma:

\begin{lemma}\label{lemma:global-to-local-relation}
If $G_\sigma$ is chosen so that it
maps $\CI_c(X;\Lambda X\oplus\Lambda X)$ to
$\CI$ forms on $\tilde X$ in a neighborhood $\cU$ of $\overline{X}$,
and so that $r_X P_\sigma G_\sigma=\Id$
where $r_X$ denotes restriction to $X$,
then for $\im\sigma\gg 1$, then
\eqref{eq:form-res-connection}, resp.\ \eqref{eq:form-res-connection-2}, hold
on $\CI_c(X;\Lambda^k X)$, resp.\ $\CI_c(X;\Lambda^{k-1} X)$, with
$G_\sigma$ replacing $P_\sigma^{-1}$.
\end{lemma}

\begin{proof}
We consider \eqref{eq:form-res-connection}; the
treatment of \eqref{eq:form-res-connection-2} is completely analogous.
Let $f\in\CI_c(X;\Lambda^k X)$, $\im\sigma\gg 1$. We claim that
$$
u=\delta_X d_X\pi_{X,T,k} JF^{-\imath\sigma+(n-2k-1)/2}
r_X G_\sigma e_X F^{\imath\sigma-(n-2k-1)/2
  -2}J^{-1}\iota_{X,T,k}f\in\CI(X;\Lambda^k X)
$$
satisfies
\begin{equation}\label{eq:hyperbolic-Lap-eqn}
\Big(-\Delta_X+\sigma^2+\Big(\frac{n-2k- 1}{2}\Big)^2\Big)u=\tilde f,\
\tilde f=\delta_X d_X f.
\end{equation}
Indeed,
using $r_X 
P_\sigma=P_\sigma r_X$,
\begin{equation*}\begin{aligned}
&\begin{bmatrix}-\Delta_X+\sigma^2+\Big(\frac{n-2k-1}{2}\Big)^2 &
-2d_X\\
2\delta_X &
-\Delta_X+\sigma^2+\Big(\frac{n-2k+3}{2}\Big)^2\end{bmatrix}\\
&\qquad\qquad\times JF^{-\imath\sigma+(n-2k-1)/2}r_X
G_\sigma e_X F^{\imath\sigma-(n-2k-1)/2 -2}J^{-1}\iota_{X,T,k} f\\
&=JF^{-\imath\sigma+(n-2k-1)/2+2}{P}_\sigma r_X
G_\sigma e_X F^{\imath\sigma-(n-2k-1)/2 -2}J^{-1}\iota_{X,T,k}
f=r_X\iota_{X,T,k}f=\begin{bmatrix}f\\ 0\end{bmatrix},
\end{aligned}\end{equation*}
so with
$$
\tilde u=JF^{-\imath\sigma+(n-2k-1)/2}
r_X G_\sigma e_X F^{\imath\sigma-(n-2k-1)/2 -2}J^{-1}\iota_{X,T,k}f,
$$
so $u=\delta_X d_X\pi_{X,T,k} \tilde u$,
one has
\begin{equation*}\begin{aligned}
\delta_X d_X f&=\delta_X d_X\pi_{X,T,k}\begin{bmatrix}f\\
  0\end{bmatrix}\\
&=
\delta_X d_X\pi_{X,T,k}\begin{bmatrix}-\Delta_X+\sigma^2+\Big(\frac{n-2k-1}{2}\Big)^2 &
-2d_X\\
2\delta_X &
-\Delta_X+\sigma^2+\Big(\frac{n-2k+3}{2}\Big)^2\end{bmatrix}\tilde u\\
&=\pi_{X,T,k}\begin{bmatrix}\delta_X d_X&0\\0&d_X\delta_X\end{bmatrix}\begin{bmatrix}-\Delta_X+\sigma^2+\Big(\frac{n-2k-1}{2}\Big)^2 &
-2d_X\\
2\delta_X &
-\Delta_X+\sigma^2+\Big(\frac{n-2k+3}{2}\Big)^2\end{bmatrix}\tilde u\\
&=\pi_{X,T,k}\begin{bmatrix}\delta_X d_X\Big(-\Delta_X+\sigma^2+\Big(\frac{n-2k-1}{2}\Big)^2\Big) &
0\\
0 &
d_X\delta_X\Big(-\Delta_X+\sigma^2+\Big(\frac{n-2k+3}{2}\Big)^2\Big)\end{bmatrix}\tilde
u\\
&=\delta_X
d_X\Big(-\Delta_X+\sigma^2+\Big(\frac{n-2k-1}{2}\Big)^2\Big)
\pi_{X,T,k}\tilde u=\Big(-\Delta_X+\sigma^2+\Big(\frac{n-2k-1}{2}\Big)^2\Big) u,
\end{aligned}\end{equation*}
and
\begin{equation*}\begin{aligned}
&u\in\pi_{X,T,k}JF^{-\imath\sigma+(n-2k-1)/2}\CI(\overline{X};\Lambda^k\tilde X\oplus\Lambda^{k-1}\tilde X)\\
&\qquad\subset
F^{-\imath\sigma+(n-2k-1)/2}\CI(\overline{X};\Lambda^k\tilde X)+
F^{-\imath\sigma+(n-2k-1)/2}\frac{dF}{F}\wedge
\CI(\overline{X};\Lambda^{k-1}\tilde X),
\end{aligned}\end{equation*}
and thus is in $L^2(X;\Lambda^k X)$.
Therefore, as given $\tilde f\in L^2(X;\Lambda^k X)$
there is a unique $L^2$ form solving \eqref{eq:hyperbolic-Lap-eqn},
we conclude that
$\Big(-\Delta_X+\sigma^2+\Big(\frac{n-2k-1}{2}\Big)^2\Big)^{-1}\delta_X
d_Xf$
is indeed given by the right hand side of the second equality in
\eqref{eq:form-res-connection}.
Since
$$
\delta_X
d_X\Big(-\Delta_X+\sigma^2+\Big(\frac{n-2k-1}{2}\Big)^2\Big)^{-1}f
$$
also solves \eqref{eq:hyperbolic-Lap-eqn} due to the fact that
$\Delta_X$ and $\delta_X d_X$
commute as operators on $\CI(X;\Lambda X)$, and as it is in $L^2$ (for
$\im\sigma\gg 1$), the first equality in
\eqref{eq:form-res-connection} also holds.
\end{proof}

Concretely, $G_\sigma$ is constructed using a complex absorption operator
$Q_\sigma$, with Schwartz kernel supported in $(\tilde X\setminus
\cU)^2$, as
$$
G_\sigma=(P_\sigma-\imath Q_\sigma)^{-1}.
$$
In fact, $P_\sigma-\imath Q_\sigma$ is
Fredholm between appropriate spaces recalled in the next section, with
a meromorphic inverse and with non-trapping high energy estimates
under non-trapping assumptions on $X$. (Technically $Q$ is defined
only for a certain set of $\sigma$, or rather one needs to use
different operators $Q$ in different subsets of $\Cx$, but in strips,
or even in somewhat larger conic sectors, which are our main interest,
a single $Q$ suffices. We refer the reader to
\cite[Section~4.7]{Vasy-Dyatlov:Microlocal-Kerr} for further details.)
Then, as the right hand side of \eqref{eq:form-res-connection} is meromorphic on
$\Cx$ with finite rank poles and has appropriate high energy estimates
under non-trapping assumptions, one obtains such an extension of the
left hand side.
A similar argument applies for \eqref{eq:form-res-connection-2}, but we
need to note this is acting on $k-1$ forms on $X$, and thus we need to
replace $k$ by $k+1$ throughout to obtain a formula for the $k$-form
Laplacian, resulting in the shift in the statement of Theorem~\ref{thm:main}.

A bit of care is needed in order to derive the precise form of the
mapping properties and the corresponding
high energy estimates. Namely, as we recall in the next sections,
$$
(P_\sigma-\imath Q_\sigma)^{-1}:H^{s-1}(\tilde X;\Lambda
\tilde X)\to 
H^{s}(\tilde X;\Lambda
\tilde X),\ s>1/2-\im\sigma,
$$ 
with the high energy estimate that for fixed $s$, $\sigma$ satisfying
$-s+1/2<\im\sigma$, $|\re\sigma|>R$, $R>0$ sufficiently large,
$$
\|(P_\sigma-\imath Q_\sigma)^{-1}\|_{\cL(H^{s-1}_{|\sigma|^{-1}}(\tilde X;\Lambda
\tilde X),
H^{s}_{|\sigma|^{-1}} (\tilde X;\Lambda
\tilde X))}\leq C|\sigma|^{-1},
$$
where $H^{s}_{|\sigma|^{-1}} (\tilde X)$ is the semiclassical Sobolev
space in which derivatives come with a prefactor of $|\sigma|^{-1}$;
see the introduction of \cite{Vasy:Microlocal-AH} for more
details. Now all the other operators in \eqref{eq:form-res-connection}-\eqref{eq:form-res-connection-2}
are straightforward to estimate, being bundle maps or differential
operators. However, these are singular maps: $F$ vanishes at $\pa X$,
and $J^{-1}$ involves $\frac{dF}{F}\wedge$ when applied to normal
forms, i.e.\ essentially $\frac{d\mu}{\mu}\wedge$. Thus, dropping the
bundles from the notation momentarily,
\begin{equation*}\begin{aligned}
&\|(P_\sigma-\imath Q_\sigma)^{-1}e_X F^{\imath\sigma-(n-2k-1)/2
  -2}J^{-1}\iota_{X,T,k}f\|_{H^s_{|\sigma|^{-1}}(\tilde X)}\\
&\qquad\qquad\leq C\|F^{\imath\sigma-(n-2k-1)/2
  -2}f\|_{H^{s-1}_{|\sigma|^{-1}}(\overline{X}_\even)},\\
&\|(P_\sigma-\imath
Q_\sigma)^{-1} e_X F^{\imath\sigma-(n-2k-1)/2-2}J^{-1}\iota_{X,N,k-1}f\|_{H^s_{|\sigma|^{-1}}(\tilde X)}\\
&\qquad\qquad\leq
C\big(\|F^{\imath\sigma-(n-2k-1)/2
  -2}f\|_{H^{s-1}_{|\sigma|^{-1}}(\overline{X}_{\even})}+\|F^{\imath\sigma-(n-2k-1)/2
  -4}d\mu\wedge f\|_{H^{s-1}_{|\sigma|^{-1}}(\overline{X}_{\even})}\big)\\
&\qquad\qquad\qquad\leq \tilde C\|F^{\imath\sigma-(n-2k-1)/2
  -4}f\|_{H^{s-1}_{|\sigma|^{-1}}(\overline{X}_{\even})},
\end{aligned}\end{equation*}
where the loss for normal forms relative to tangential forms (in terms of a simple Sobolev space, given on the right hand side
of the last inequality)
comes from the singular factor in $J$ giving rise to the $d\mu\wedge$
term, and where the spaces on
$\tilde X$ and $\overline{X}_{\even}$ are the sections of appropriate
degree parts of $\Lambda\tilde X\oplus\Lambda\tilde X$, resp.\
$\Lambda \overline{X}_{\even}$. Further,
\begin{equation*}\begin{aligned}
&\|F^{\imath\sigma-(n-2k-1)/2+2}\pi_{X,T,k}
JF^{-\imath\sigma+(n-2k-1)/2}r_X
v\|_{H^s_{|\sigma|^{-1}}(\overline{X}_{\even})}\\
&\qquad\leq C'\big(\|F^{\imath\sigma-(n-2k-1)/2+2}\pi_{X,T,k} 
JF^{-\imath\sigma+(n-2k-1)/2}r_X
v\|_{H^s_{|\sigma|^{-1}}(\overline{X}_{\even})}\\
&\qquad\qquad\qquad+
\|F^{\imath\sigma-(n-2k-1)/2} d\mu\wedge \pi_{X,T,k} 
JF^{-\imath\sigma+(n-2k-1)/2}r_X
v\|_{H^s_{|\sigma|^{-1}}(\overline{X}_{\even})}\big)\\
&\qquad\leq C'' \|v\|_{H^s_{|\sigma|^{-1}}(\tilde X)},\\
&\|F^{\imath\sigma-(n-2k-1)/2}\pi_{X,N,k-1} JF^{-\imath\sigma+(n-2k-1)/2}r_Xv\|_{H^s_{|\sigma|^{-1}}(\overline{X}_{\even})}\leq C' \|v\|_{H^s_{|\sigma|^{-1}}(\tilde X)},
\end{aligned}\end{equation*}
where the loss is now in tangential forms due to $J$. However, these losses
are merely apparent, as we momentarily show using the special structure
of $d_X\delta_X$ and $\delta_X d_X$.

Indeed, $\delta_X d_X, d_X\delta_X$ are even differential operators, i.e.\ when
regarded as an operator on $\overline{X}_{\even}$, they satisfy
$\delta_X d_X,d_X\delta_X\in\Diff^2(\overline{X}_{\even};\Lambda
\overline{X}_{\even})$, and even $\delta_X
d_X,d_X\delta_X\in\Diffb^2(\overline{X}_{\even};\Lambda \overline{X}_{\even})$; this can be seen from
a direct calculation, which we discuss below in the general
conformally compact case in Lemma~\ref{lemma:d-del-even}. (Recall that
$\Vb(\overline{X}_{\even})$ is the set of smooth vector fields tangent
to the boundary; $\Diffb$ is generated by these.) In fact, an even
stronger statement also holds for certain parts of this operator,
namely, with $d\mu\wedge$ denoting the operator of wedge product with $d\mu$,
$$
(d\mu\wedge)d_X\delta_X,\delta_X d_X(d\mu\wedge)\in\mu\Diffb^2(\overline{X}_{\even};\Lambda
\overline{X}_{\even});
$$
see Lemma~\ref{lemma:d-del-even}.
Correspondingly, for any $\alpha\in\Cx$, basically relying on
$$
\mu^{\alpha/2}(\mu\pa_\mu)\mu^{-\alpha/2}=\mu\pa_\mu-\alpha/2,
$$
one has
$$
F^{\alpha}\delta_X
d_XF^{-\alpha}\in\Diffb^2(\overline{X}_{\even};\Lambda
\overline{X}_{\even})\subset
\Diff^2(\overline{X}_{\even};\Lambda \overline{X}_{\even}),
$$
and
$$
F^{\alpha}\delta_X
d_XF^{-\alpha}(\frac{dF}{F}\wedge)\in\Diffb^2(\overline{X}_{\even};\Lambda
\overline{X}_{\even})\subset
\Diff^2(\overline{X}_{\even};\Lambda \overline{X}_{\even}),
$$
with analogous statements for $d_X\delta_X$.
Thus,
\begin{equation*}\begin{aligned}
&F^{\imath\sigma-(n-2k-1)/2}\delta_Xd_X\pi_{X,T,k}
JF^{-\imath\sigma+(n-2k-1)/2}\in\Diffb^2(\overline{X}_{\even};\Lambda
\overline{X}_{\even}),\\
&F^{\imath\sigma-(n-2k-1)/2} d_X\delta_X\pi_{X,N,k-1}
JF^{-\imath\sigma+(n-2k-1)/2}\in\Diffb^2(\overline{X}_{\even};\Lambda
\overline{X}_{\even}),\\
&F^{\imath\sigma-(n-2k-1)/2-2} (d\mu\wedge )d_X\delta_X\pi_{X,N,k-1}
JF^{-\imath\sigma+(n-2k-1)/2}\in\Diffb^2(\overline{X}_{\even};\Lambda
\overline{X}_{\even}).\\
\end{aligned}\end{equation*}
Therefore,
for $s\in\RR$, the operators
\begin{equation*}\begin{aligned}
&F^{\imath\sigma-(n-2k-1)/2}\langle|\sigma|\rangle^{-2}\delta_Xd_X\pi_{X,T,k}
JF^{-\imath\sigma+(n-2k-1)/2},\\
&F^{\imath\sigma-(n-2k-1)/2} \langle|\sigma|\rangle^{-2} d_X\delta_X\pi_{X,N,k-1}
JF^{-\imath\sigma+(n-2k-1)/2},\\
&F^{\imath\sigma-(n-2k-1)/2-2} \langle|\sigma|\rangle^{-2}
(d\mu\wedge)
d_X\delta_X\pi_{X,N,k-1}
JF^{-\imath\sigma+(n-2k-1)/2}\\
\end{aligned}\end{equation*}
are uniformly bounded in $\cL(H^{s+2}_{|\sigma|^{-1}}(\overline{X}_{\even};\Lambda
\overline{X}_{\even}),H^{s}_{|\sigma|^{-1}}(\overline{X}_{\even};\Lambda
\overline{X}_{\even}))$. In summary, using
\eqref{eq:form-res-connection}-\eqref{eq:form-res-connection-2} for $\im\sigma>-s-3/2$, $|\re\sigma|$
sufficiently large,
\begin{equation*}\begin{aligned}
&\|F^{\imath\sigma-(n-2k-1)/2}\delta_X
d_X\Big(-\Delta_X+\sigma^2+\Big(\frac{n-2k-1}{2}\Big)^2\Big)^{-1}f\|_{H^s_{|\sigma|^{-1}}(\overline{X}_{\even})}\\
&\qquad\qquad\leq C_0 |\sigma|\|F^{\imath\sigma-(n-2k-1)/2
  -2}f\|_{H^{s+1}_{|\sigma|^{-1}}(\overline{X}_\even)},\\
&\|F^{\imath\sigma-(n-2k-1)/2} d_X\delta_X
\Big(-\Delta_X+\sigma^2+\Big(\frac{n-2k+3}{2}\Big)^2\Big)^{-1}f\|_{H^s_{|\sigma|^{-1}}(\overline{X}_{\even})}\\
&\qquad+\|F^{\imath\sigma-(n-2k-1)/2-2} (d\mu\wedge)d_X\delta_X
\Big(-\Delta_X+\sigma^2+\Big(\frac{n-2k+3}{2}\Big)^2\Big)^{-1}f\|_{H^s_{|\sigma|^{-1}}(\overline{X}_{\even})}\\
&\qquad\qquad\leq C_0 |\sigma|
\|F^{\imath\sigma-(n-2k-1)/2
  -2}f\|_{H^{s+1}_{|\sigma|^{-1}}(\overline{X}_{\even})}+\|F^{\imath\sigma-(n-2k-1)/2
  -4}d\mu\wedge f\|_{H^{s+1}_{|\sigma|^{-1}}(\overline{X}_{\even})}.
\end{aligned}\end{equation*}

Since $\delta_X d_X+d_X\delta_X=\Delta_X$, combining
\eqref{eq:form-res-connection}-\eqref{eq:form-res-connection-2}
gives the meromorphic continuation of
$\Big(-\Delta_X+\sigma^2+\Big(\frac{n-2k-1}{2}\Big)^2\Big)^{-1}$
itself, but with another branch arising from closed forms, i.e.\ the
meromorphic continuation is not merely to the Riemann surface of the
inverse function of $\lambda\mapsto
\sqrt{\lambda-\Big(\frac{n-2k-1}{2}\Big)^2}$, rather the joint Riemann
surface of this and $\lambda\mapsto
\sqrt{\lambda-\Big(\frac{n-2k+1}{2}\Big)^2}$. Further, what one actually obtains is
\begin{equation*}\begin{aligned}
&\Delta_X
\Big(-\Delta_X+\sigma^2+\Big(\frac{n-2k-1}{2}\Big)^2\Big)^{-1}\\
&=-\Id+\Big(\sigma^2+\Big(\frac{n-2k-1}{2}\Big)^2\Big) \Big(-\Delta_X+\sigma^2+\Big(\frac{n-2k-1}{2}\Big)^2\Big)^{-1},
\end{aligned}\end{equation*}
and thus an infinite rank pole is allowed at points where the analytic
continuation of $\lambda\mapsto\lambda$ vanishes
(note that $\lambda=\sigma^2+\Big( \frac{n-2k-1}{2}\Big)^2$ is the
spectral parameter in the above formula.)
We write
$$
\cR_X(\sigma)
$$
for the meromorphic continuation of
$\Big(-\Delta_X+\sigma^2+\Big(\frac{n-2k-1}{2}\Big)^2\Big)^{-1}$.

We
remark that with slightly more work the `cross terms', i.e.\ $\delta_X
d_X$ with normal forms and $d_X\delta_X$ with tangential forms can
also be analyzed, and then the nature of the possible pole at zero can be
described more precisely, but this is not our focus here. The basic point is
that for an operator mapping between direct sums of Banach spaces,
$\cX_0\oplus\cX_1\to\cY_0\oplus \cY_1$, if $D:\cX_1\to\cY_1$ is
invertible then the invertibility of
$A-BD^{-1}C:\cX_0\to\cY_0$ and $\begin{bmatrix}
  A&B\\C&D\end{bmatrix}$ are equivalent, with
\begin{equation}\label{eq:block-matrix-inv}
\begin{bmatrix}
  A&B\\C&D\end{bmatrix}^{-1}
=\begin{bmatrix}(A-BD^{-1}C)^{-1} &-(A-BD^{-1}C)^{-1} BD^{-1}\\
-D^{-1}C (A-BD^{-1}C)^{-1} &D^{-1}+D^{-1}C (A-BD^{-1}C)^{-1}  BD^{-1}\end{bmatrix}.
\end{equation}
There is an analogous formula if the role of the two components are
interchanged. Since for top and bottom degree forms one of the two
components is trivial (as $\Lambda^{-1}X$, resp.\ $\Lambda^{n+1}X$ are
trivial), one can proceed inductively from the two extremes towards
middle degrees. Thus, for $1$-forms, for instance, one uses that one has
obtained $D_\sigma$ on $0$-forms to conclude that, provided that the
domains remain compatible, one has a meromorphic continuation for
1-forms with at most a finite rank pole at $0$ since, writing the right
hand side of \eqref{eq:block-matrix-inv} as $\begin{bmatrix}E&F\\G&H\end{bmatrix}$,
$A^{-1}=E-F(D-CA^{-1}B)G$, and $CA^{-1}B$ (recall that $B$ and $C$ are
$-2d_X$ and $2\delta_X$) can be computed using the
information already obtained above, including at $0$.

\section{Conformally compact spaces}\label{sec:conformal}
We now extend the results to general even conformally compact spaces.
That is, if $(X,g)$
is Riemannian and even asymptotically hyperbolic,
there is a product decomposition near the
boundary $Y_{{y}}$ of $\overline{X}$
such that
$$
g=\frac{dx^2+\tilde h(x,{y},d{y})}{x^2},
$$
with $\tilde h$ even in $x$, i.e.\ $\tilde h=h(x^2,{y},d{y})$,
with $h$ smooth. We write $\overline{X}_{\even}$ for $\overline{X}$
with the new smooth structure in which $\mu=x^2$ is a boundary
defining function. We consider $h$ as a symmetric 2-cotensor on $Y$ valued
function on $\overline{X}_{\even}$ defined near $Y$.
Before considering the appropriate extension of an operator related to
the spectral family of $\Delta_X$ across $\pa X$, we first discuss
$d_X\delta_X$ and $\delta_X d_X$ in some detail.

\begin{lemma}\label{lemma:d-del-even}
Suppose that $(X,g)$ is equipped with an even asymptotically
hyperbolic metric, with $\overline{X}_{\even}$ being the
compactification equipped with the even smooth structure. Then
$$
\delta_X d_X,d_X\delta_X\in\Diffb^2(\overline{X}_{\even};\Lambda
\overline{X}_{\even}).
$$
Further, with $d\mu\wedge$ denoting the operator of wedge product with
$d\mu$,
$$
(d\mu\wedge)d_X\delta_X,\delta_X d_X(d\mu\wedge)\in\mu\Diffb^2(\overline{X}_{\even};\Lambda
\overline{X}_{\even}).
$$
\end{lemma}

\begin{proof}
With $\mu=x^2$, we use a conormal vs. tangential decomposition of
forms near $Y$ on $\overline{X}_{\even}$, i.e.\ we write $k$-forms as
linear combinations of
$$
dy^\alpha,\ d\mu\wedge dy^\beta,\ |\alpha|=k,\ |\beta|=k-1.
$$
In this basis, $d_X$ has the form
$$
d_X=\begin{bmatrix}d_Y&0\\\pa_\mu&-d_Y\end{bmatrix},
$$
while $g$ has the form $g=\frac{d\mu^2}{4\mu^2}+\frac{h}{\mu}$, so the
dual metric is $G=4\mu^2\pa_\mu^2+\mu H$, where $H$ is the dual metric
of $h$. Correspondingly,
$$
|dg|=\frac{1}{2\mu^{(n+1)/2}}\,d\mu\,dh=\frac{\sqrt{\det h}}{2\mu^{(n+1)/2}}\,d\mu\,dh,
$$
and on $k$-forms the dual metric is
$$
G_k=\begin{bmatrix}\mu^k H_k&0\\0&4\mu^{k+1}H_{k-1}\end{bmatrix},
$$
where $H_k$ is the dual metric of $h$ on boundary $k$-forms. We
compute $\delta_X$ as
$\delta_X=G_{k-1}^{-1} d_{\mathrm{base}}^*G_k$, $d_{\mathrm{base}}$ being the
adjoint of $d_X$ where the Euclidean inner product is used in the
fibers of $T^*X$ via a local trivialization, but the metric density
$|dg|$ is used to integrate, i.e.\ $d_{\mathrm{base}}^*=(\det
g)^{-1/2}\delta_{\RR^n}(\det g)^{1/2}$. This gives
\begin{equation*}\begin{aligned}
&\delta_{X,k}=\begin{bmatrix}\mu\delta_Y&-4\mu^2\pa_\mu+2\mu(n-2k-1)+\mu^2\gamma\\
0&-\mu\delta_Y\end{bmatrix},\\
& \gamma=-4H_{k-1}^{-1}(\det h)^{-1/2}\pa_\mu
(\det h)^{1/2} H_{k-1}.
\end{aligned}\end{equation*}
This yields
\begin{equation*}\begin{aligned}
&d_{X,k-1}\delta_{X,k}=\begin{bmatrix}\mu
  d_Y\delta_Y&d_Y(-4\mu^2\pa_\mu+2\mu(n-2k-1)+\mu^2\gamma)\\
  \pa_\mu\mu\delta_Y&\pa_\mu
  (-4\mu^2\pa_\mu+2\mu(n-2k-1)+\mu^2\gamma)+\mu
  d_Y\delta_Y\end{bmatrix},\\
&\delta_{X,k+1} d_{X,k}\\
&=\begin{bmatrix}\mu\delta_Y
  d_Y+(-4\mu^2\pa_\mu+2\mu(n-2k-3)+\mu^2\gamma)\pa_\mu&-(-4\mu^2\pa_\mu+2\mu(n-2k-3)+\mu^2\gamma)d_Y\\
-\mu\delta_Y\pa_\mu&\mu\delta_Y d_Y\end{bmatrix},
\end{aligned}\end{equation*}
which are indeed in $\Diffb^2(\overline{X}_{\even};\Lambda
\overline{X}_{\even})$. Furthermore,
\begin{equation*}\begin{aligned}
&(d\mu\wedge)d_{X,k-1}\delta_{X,k}=\begin{bmatrix}\mu
  d_Y\delta_Y&d_Y(-4\mu^2\pa_\mu+2\mu(n-2k-1)+\mu^2\gamma)\\
0&0\end{bmatrix},\\
&\delta_{X,k+1} d_{X,k}(d\mu\wedge)
=\begin{bmatrix}0&-(-4\mu^2\pa_\mu+2\mu(n-2k-3)+\mu^2\gamma)d_Y\\
0&\mu\delta_Y d_Y\end{bmatrix},
\end{aligned}\end{equation*}
which are in $\mu \Diffb^2(\overline{X}_{\even};\Lambda
\overline{X}_{\even})$, completing the proof.
\end{proof}

Note that this in particular implies that $\Delta_X\in\Diffb^2(\overline{X}_{\even};\Lambda
\overline{X}_{\even})$. In the scalar setting, more is true: after one
conjugates the spectral family, $\Delta_X-\sigma^2-(n-1)^2/4$, by the
appropriate power of $\mu$, one can factor out $\mu$ and still have
a differential operator with smooth coefficients. The appropriate
power is closely related to the asymptotic behavior of the Green's
function at $\pa X$. The diverse behavior of the form Laplacian on
different kinds of forms makes this a more difficult process in the
form valued setting. For instance, notice that one can factor $\mu$
out of $d_{X,k-1}\delta_{X,k}$ on the right, while for $\delta_{X,k+1}
d_{X,k}$ this can be done on the left -- and this ignores additional
issues from the spectral family!

We now
describe two possible ways of proceeding, with the first being an
analogue of \cite{Vasy:Microlocal-AH} but working with an extended
system (not merely extending a form bundle, but working with two copies);
we pursue the second one of these in detail, which is based on the Minkowski
space model.

The first method is as follows.
One may regard \eqref{eq:relate-Mellin} as a statement that the right
hand side, valid in $\HH^n=X$, extends to a differential operator on
$\sphere^n=\tilde X$ of the appropriate type, with smooth coefficients, acting
on two copies of the form bundle on the sphere, after
$J^{-1}F^{\imath\sigmat-2}$ is applied from the left,
$F^{-\imath\sigmat}J$ applied from the right, {\em and the smooth structure
is changed to the smooth structure corresponding to the boundary
defining function $F^2$}. In view of \eqref{eq:hyp-MT-form}, this is a
statement about the spectral family of a slightly modified version of
$\Delta_X$, incorporated into a system. This transformation only
depends on a choice of $F$, and for most purposes the only relevant
feature of $F$ is that it is a boundary defining function,
well-behaved relative to the evenness statement. That is, if $(X,g)$
is Riemannian and even, there is a product decomposition near the
boundary $Y_{{y}}$ of $\overline{X}$
such that
$$
g=\frac{dx^2+\tilde h(x,{y},d{y})}{x^2},
$$
with $\tilde h$ even in $x$, i.e.\ $\tilde h=h(x^2,{y},d{y})$.
Taking $F=x$, one modifies the system
\begin{equation}\begin{aligned}
&{\tilde P}_\sigma=\begin{bmatrix}-\Delta_X+\sigma^2+\Big(\frac{n-2k-1}{2}\Big)^2 &
-2d_X\\
2\delta_X &
-\Delta_X+\sigma^2+\Big(\frac{n-2k+3}{2}\Big)^2\end{bmatrix}\in \Diff^2(X;\Lambda^k X\oplus\Lambda^{k-1} X)
\end{aligned}\end{equation}
to the operator
\begin{equation}
{P}_\sigma|_{X_{\even}}=J^{-1}F^{\imath\sigma-(n-2k-1)/2-2}{\tilde P}_\sigma F^{-\imath\sigma+(n-2k-1)/2}J,
\end{equation}
which one now checks is the restriction of an operator ${P}_\sigma$ defined on an
extension $\tilde X$ of $\overline{X}_{\even}$ across $Y$, and satisfying the
requirements of \cite{Vasy-Dyatlov:Microlocal-Kerr} and
\cite{Vasy:Microlocal-AH}. This was checked explicitly on functions in
\cite{Vasy:Microlocal-AH}. Note that at the level of the principal
symbol, given by the dual metric {\em function} (times the identity
operator),
this means that $F^{-2}G$
extends smoothly to $T^*\tilde X$, which is automatic for an
even asymptotically hyperbolic metric.

A different way of proceeding, which we pursue instead, is via extending the metric to an
ambient metric, playing the role of the Minkowski metric, which is
homogeneous of degree $-2$. Thus, one considers
$M=\RR^+_{\rho}\times\tilde X$, as well as $\RR^+_r\times X$, with
$r=F\rho$, $F=x$, although we note that while with $F$ defined above
in the Minkowski setting, the hyperbolic metric has some higher order
(in $x$)
$dx^2$ terms in view of \eqref{eq:hyp-metric-form-F},
which however do not affect properties of the extension. On
$\RR^+_r\times X$ the analogue of the Minkowski metric is
$$
\tilde g=dr^2-r^2 g=r^2\Big(\frac{dr^2}{r^2}-g\Big)=\rho^2\Big(F^2\Big(\frac{d\rho}{\rho}+\frac{dF}{F}\Big)^2-F^2g\Big).
$$
Substituting the form of $g$ and writing $F=x$, $F^2=\mu$,
$$
\tilde g=\rho^2\Big(\mu
\frac{d\rho^2}{\rho^2}+\frac{1}{2}\Big(\frac{d\rho}{\rho}\otimes
d\mu+d\mu\otimes \frac{d\rho}{\rho}\Big)-h(\mu,{y},d{y})\Big).
$$
But now the desired extension is immediate to a neighborhood of
$\overline{X}_{\even}$ in $\tilde X$ (which is
all that is required for the analysis), by simply extending $h$
smoothly to a neighborhood.
This is easily checked to be Lorentzian (and as for this part forms
are irrelevant, there is nothing to check beyond what was done in the
scalar setting
in \cite{Vasy:Microlocal-AH}), with $d\mu$ time-like in $\mu<0$, and now the Mellin transform
gives rise to a smooth family of operators ${P}_\sigma$ on $\tilde X$, related to
${\tilde P}_\sigma$ via the same procedure as in the Minkowski
setting. (In the scalar setting, this is a special case of metrics
currently under study by Baskin, Wunsch and the author \cite{Baskin-Vasy-Wunsch:Radiation}, termed
`scattering Lorentzian metrics'.) Since the requirements for the
analysis involve the principal symbol for the Mellin transform
(including in the high energy sense), which is the same as in the
scalar setting (times the identity), namely the dual metric function
on $M$, with $\sigma$ being the Mellin-dual variable of
$\rho\pa_\rho$,
plus {\em some} bound on the
subprincipal symbol at $N^*Y$ as a bundle endomorphism (which is
automatic by the compactness of $Y$), the results of \cite{Vasy-Dyatlov:Microlocal-Kerr} and
\cite{Vasy:Microlocal-AH} are now applicable. Note that the
$\sigma$-dependence of the subprincipal symbol can be read off from
the b-principal symbol of $\Box_{\tilde g}$, so the issue is finding a
$\sigma$-independent constant (which, again, at most shifts by a
constant what spaces should be used).

However, it is actually instructive to compute the subprincipal symbol
at $N^*Y$. It turns out that this is a scalar bundle map on $\Lambda^k
\tilde X\oplus\Lambda^{k-1} \tilde X$. First, the dual metric of $\tilde g$ is
$$
\tilde
G=\frac{2}{\rho}(\pa_\mu\otimes\pa_\rho+\pa_\rho\otimes\pa_\mu)-\frac{4\mu}{\rho^2}\pa_\mu^2-\rho^{-2}
H,
$$
with $H$ the dual metric of $h$, and the metric density is
$$
|d\tilde g|=\frac{\rho^n}{2}\,d\rho\,d\mu\,|dh|
=\frac{\rho^n}{2}\,\sqrt{|\det h|}\,d\rho\,d\mu\,dy.
$$
Next, writing
$k$-forms on $M$ as
linear combinations of
$$
dy^{\alpha},\ d\mu\wedge dy^\beta,\ d\rho\wedge
dy^\gamma,\ d\rho\wedge d\mu\wedge dy^\delta,\
|\alpha|=k,\ |\beta|=k-1=|\gamma|,\ |\delta|=k-2,
$$
one obtains that on $k$-forms
$$
d=\begin{bmatrix} d_Y&0&0&0\\\pa_\mu&-d_Y&0&0\\
\pa_\rho&0&-d_Y&0\\0&\pa_\rho&-\pa_\mu&d_Y\end{bmatrix}
$$
and thus, using the expression for $d$ on $k-1$-forms to compute its
adjoint on $k$-forms,
\begin{equation}\label{eq:delta-k}
\delta_k=\rho^{-2}\begin{bmatrix}\delta_Y&-4(\pa_{\mu})_h^*+2\rho(-\pa_\rho-\frac{n-2k+1}{\rho})&2\rho(\pa_\mu)^*_h&0\\0&-\delta_Y&0&2\rho(\pa_\mu)^*_h\\
0&0&-\delta_Y&4\mu(\pa_\mu)^*_h-2\rho(-\pa_\rho-\frac{n-2k+2}{\rho})\\0&0&0&\delta_Y\end{bmatrix},
\end{equation}
where
$$
(\pa_\mu)^*_h=-\frac{1}{\sqrt{\det h}} h\pa_\mu \sqrt{\det h} H.
$$
This computation is analogous to the computation of $d_X\delta_X$ and
$\delta_X d_X$ above, but is more complicated as one needs to work
with a four-by-four system. It can again be done in steps, first computing the adjoint
$$
d^*_{\mathrm{base}}=\rho^{-n}(\det h)^{-1/2}\delta_{\RR^n}(\det
h)^{1/2}\rho^n
$$
of
$d$ relative to the Euclidean inner product on the fibers of $T^*M$ in
local coordinates $(\rho,\mu,y)$ but with the actual metric density
in the base, which is straightforward, and then computing
$\delta=\tilde G^{-1}d^*_{\mathrm{base}}\tilde G$, where $\tilde G$
also stands for the dual metric on the form bundle, which is a block
matrix of the form
$$
\tilde G_k=
\begin{bmatrix} \rho^{-2k}(-H)_k&0&0&0\\
0&-\frac{4\mu}{\rho^2}\rho^{-2(k-1)}(-H)_{k-1}&\frac{2}{\rho}\rho^{-2(k-1)} (-H)_{k-1}&0\\
0&\frac{2}{\rho}\rho^{-2(k-1)} (-H)_{k-1}&0&0\\0&0&0&-\frac{4}{\rho^2}\rho^{-2(k-2)}(-H)_{k-2}\end{bmatrix}
$$
on $k$-forms, where $(-H)_j$ is the inner product induced by $-H$ on
$j$-forms on $Y=\pa X$. Note that the $k$-dependent powers of $\rho$
arise from the degree of the form in the $y$-variables. Thus,
$\delta_k=\tilde G_{k-1}^{-1}d^*_{\mathrm{base}}\tilde G_k$ gives rise
to \eqref{eq:delta-k}.

Now one can compute $\Delta_k=d_{k-1}\delta_k+\delta_{k+1}d_k$ in a
straightforward, if computationally slightly messy, manner. To state
the result of the computation, it is convenient to rewrite
$k$-forms on $M$ as
linear combinations of
$$
dy^{\alpha},\ d\mu\wedge dy^\beta,\ \frac{d\rho}{\rho}\wedge
dy^\gamma,\ \frac{d\rho}{\rho}\wedge d\mu\wedge dy^\delta,\
|\alpha|=k,\ |\beta|=k-1=|\gamma|,\ |\delta|=k-2.
$$
Then one obtains that, with $\Vb(\tilde X;Y)$ denoting set of vector
fields on $\tilde X$ tangent to $Y$, $\Diffb^m(\tilde X;Y)$ denoting
finite products up to $m$ factors of these, and $\Diffb^m(\tilde
X;Y;E)$ the corresponding operators acting on sections of a vector
bundle $E$ on $\tilde X$ (with the action defined via trivialization
as matrices of scalar operators),
\begin{equation*}\begin{aligned}
&\cM_\rho\rho^2\Box_{\tilde
  g}\cM_\rho^{-1}=(4\pa_\mu\mu\pa_\mu-4(\imath\tilde\sigma+(n-2k-1)/2)\pa_\mu)\otimes\Id+\tilde
Q,\\
&\qquad \tilde Q\in \Diffb^2(\tilde X;Y;\Lambda^k\tilde X\oplus\Lambda^{k-1}\tilde X),
\end{aligned}\end{equation*} 
or
\begin{equation*}\begin{aligned}
{P}_\sigma&=\cM_\rho\rho^2\rho^{(n-2k-1)/2}\Box_{\tilde
  g}\rho^{-(n-2k-1)/2}\cM_\rho^{-1}\\
&=(4\pa_\mu\mu\pa_\mu-4\imath\sigma\pa_\mu)\otimes\Id+Q,
\ Q\in \Diffb^2(\tilde X;Y;\Lambda^k\tilde X\oplus\Lambda^{k-1}\tilde X).
\end{aligned}\end{equation*}
This means that the spaces for Fredholm analysis, briefly recalled
below from \cite{Vasy:Microlocal-AH}, are
\begin{equation*}\begin{aligned}
&{P}_\sigma:\cX^s\to\cY^{s-1},\\
&\cX^s=\{u\in H^s(\tilde
X;\Lambda^k\tilde X\oplus \Lambda^{k-1}\tilde X):\ {P}_\sigma u\in
H^{s-1}(\tilde
X;\Lambda^k\tilde X\oplus \Lambda^{k-1}\tilde X)\},\\
&\cY^{s-1}=H^{s-1}(\tilde
X;\Lambda^k\tilde X\oplus \Lambda^{k-1}\tilde X),\ s>-\im\sigma+1/2,
\end{aligned}\end{equation*}
and elements of the distributional kernel of ${P}_\sigma$ behave
as $(\mu\pm \imath 0)^{\imath\sigma}$; these are just outside the space
$\cX^s$ when $s>-\im\sigma+1/2$. As mentioned before, $P_\sigma^{-1}$,
or rather $(P_\sigma-\imath Q_\sigma)^{-1}$, where $Q_\sigma$ is the
complex absorbing operator,
is related to the resolvent family of $\Delta_X$ via the same
procedure as in the
Minkowski setting; no special properties of the Minkowski metric were
used in the proof of \eqref{eq:form-res-connection}-\eqref{eq:form-res-connection-2}.

\section{Analysis}\label{sec:analysis}
We finally recall the analytic set-up from \cite{Vasy-Dyatlov:Microlocal-Kerr} and
\cite{Vasy:Microlocal-AH} to complete the picture.
Here we ${P}_\sigma$ is exactly the family of operators constructed in
the previous section.

The key part is estimates in a strip
$|\im\sigma|<C$, which means that even in the large parameter sense
the principal symbol of $P_\sigma$ is a real scalar. More precisely,
the general setup, satisfied by our operator $P_\sigma$,
is that $P_\sigma$, of order $m$, has real scalar principal
symbol even with $\sigma$ as a large parameter (even if $\sigma$ is
complex but is in a strip, in the principal symbol sense it may be
regarded real, and we often do so for convenience),
i.e.\ the principal symbol is $p_\sigma \Id$, with $p_\sigma$ real
valued. The classical principal symbol (without $\sigma$ as a large
parameter) is denoted by $p$, and is assumed to be independent of
$\sigma$.
It is convenient to rescale the problem to a semiclassical one
for large parameter issues, i.e.\ consider
$P_{\semi,z}=h^2P_{h^{-1}z}$, with $|\im z|<Ch$, i.e.\ at the
principal symbol level $z$ is real; the semiclassical principal symbol
is $p_{\semi,z}$.

Next, we consider the characteristic set $\Sigma$ of $p$; one assumes
that this is a union of disjoint sets $\Sigma_+,\Sigma_-$, each of
which is a union of connected components of $\Sigma$. Due to
H\"ormander's theorem \cite{Hormander:Existence}, \cite{FIOII}, one has
real principal type propagation where $p$ is not {\em radial}, i.e.\
the Hamilton vector field $\sH_p$ is not a multiple of the radial
vector field (the generator of dilations on the fibers of $T^*\tilde
X\setminus o$). One assumes (though a more general setting is
discussed in \cite{Vasy-Dyatlov:Microlocal-Kerr}; this is needed there since the
conormal bundle of the event horizon in Kerr-de Sitter space is not
radial, though it is an invariant Lagrangian submanifold) that the set
of radial points is
a union of conic Lagrangian submanifolds;
in this case under a non-degeneracy assumption it is
automatically a source or sink for the Hamilton flow within
$\Sigma_\pm$.

At radial points, the basic theorem due to Melrose in
asymptotically Euclidean scattering \cite{RBMSpec}, proved in this
generality by the author in \cite{Vasy-Dyatlov:Microlocal-Kerr}, and
refined by Haber and the author \cite{Haber-Vasy:Radial}. The result
states that if one has a solution $u$ of $P_\sigma u=f$, and $u$ possesses a
priori regularity beyond a threshold level at the radial set, then one has hyperbolic
type estimates (loss of one derivative relative to elliptic
estimates), i.e.\ $u$ is $m-1$ Sobolev orders more regular than $f$. Note that
there is no need to assume that $u$ has this $m-1$ order improved
regularity anywhere, unlike for real principal type propagation, where
one can merely propagate such estimates. On the other hand, below a
threshold level, one has the real principal type result in that
without having to assume any regularity on $u$ at the radial set, one
can propagate regularity (up to, i.e.\ below, this threshold) from a punctured
neighborhood of the radial set to the radial set, up to $m-1$ order
improved relative to $f$. Such results are local to each component of
the radial set, and indeed can be localized even within the radial
set, as shown in \cite{Haber-Vasy:Radial}. Here the threshold value is
by no means mysterious; if $P_\sigma$ is formally self-adjoint, it is
$(m-1)/2$. In general it is given by $(m-1)/2$ plus the ratio of the imaginary (or
skew-adjoint)
part of the subprincipal symbol and the Hamilton vector field applied
to the logarithm of a positive homogeneous degree one function
evaluated at the Lagrangian when the subprincipal symbol is scalar
(but possibly variable), if it is not scalar, one needs to take an
operator bound of the skew-adjoint part of the subprincipal symbol as
a self-adjoint operator. In the present case, the shift is
$-\im\sigma$; as $m=2$, this gives a threshold value of $1/2-\im\sigma$.

Finally we introduce complex absorption. This
is a pseudodifferential operator $Q_\sigma$, with real scalar
principal symbol $q$, and one considers $P_\sigma-\imath
Q_\sigma$. Here $q$ and $q_{\semi,z}$ are supported away from the
radial sets of $p$, and they are harmless in the elliptic set of $p$
and $p_{\semi,z}$. In the real
principal type region $Q$ breaks down the symmetry of the propagation
estimates (forward vs.\ backwards); for $q\geq 0$ one can propagate
estimates forwards, for $q\leq 0$ backwards.
Of course, adding $Q$ changes the operator, so we want $Q$ to be
supported outside the region we care about (such as
$\overline{X}_{\even}$ above).

Now, in order to have a Fredholm problem we need that all
bicharacteristics of $p$ in $\Sigma_\pm$ are non-trapped, i.e.\ that they
escape both in the forward and in the backward directions to locations
where they can be controlled, i.e.\ either they enter $\{q\neq 0\}$ in
finite time, or they tend to $\Lambda_\pm$. More concretely, if we
label $\Sigma_\pm$ so that
$\Lambda_+$ is a source and $\Lambda_-$ is a sink, which is the
labelling of \cite{Vasy:Microlocal-AH} (and the opposite of the
labelling of \cite{Vasy-Dyatlov:Microlocal-Kerr}) then we require that
each bicharacteristic in $\Sigma_+\setminus\Lambda_+$ tends to either $\Lambda_+$ or
enters $\{q>0\}$ in finite time in the backward direction, and enters
$\{q>0\}$ in finite time in the forward direction, while
each bicharacteristic in $\Sigma_-\setminus\Lambda_-$ tends to either $\Lambda_-$ or
enters $\{q<0\}$ in finite time in the forward direction, and enters
$\{q<0\}$ in finite time in the backward direction. (Note that in
$\{q>0\}$ and $\{q<0\}$ the requirements are automatically satisfied!)
Thus, in high regularity spaces (with $s$ bigger than a threshold) we
can propagate estimates away from $\Lambda_+\cup\Lambda_-$ (and
towards the support of the complex absorption), while in the low
regularity spaces we can proceed in the opposite direction.

Thus, if
$s$ is greater than the threshold value at
$\Lambda_+$ and $\Lambda_-$,
then one can propagate regularity and estimates from
$\Lambda_+\cup\Lambda_-$ to $\{q>0\}\cup\{q<0\}$.
For the adjoint operator under these assumptions one
has a similar result if one works with low regularity spaces, namely
if
one replaces $s$ by $-s+(m-1)$, which is
exactly the relevant space for duality arguments; one then propagates
the estimates {\em in the opposite direction}. Concretely, one has
estimates
$$
\|u\|_{H^s}\leq C(\|(P_\sigma-\imath Q_\sigma) u\|_{H^{s-m+1}}+\|u\|_{H^{-N}})
$$ 
and
$$
\|u\|_{H^{-s+m-1}}\leq C(\|(P_\sigma^*+\imath Q_\sigma^*) u\|_{H^{-s}}+\|u\|_{H^{-N'}})
$$
for appropriate $N,N'$ with compact inclusion into the spaces on the
right hand side, yielding that, with
$$
\cY^{s}=H^{s},\ \cX^s=\{u\in H^s:\ (P_\sigma-\imath Q_\sigma) u\in H^{s-m+1}\}
$$
(note that the last statement in the definition of $\cX^s$ depends on
the principal symbol of $P_\sigma-\imath Q_\sigma$ only, which is independent of
$\sigma$),
$$
P_\sigma:\cX^s\to\cY^{s-m+1},\ P_\sigma^*:\cX^{-s+m+1}\to\cY^{-s}
$$
are Fredholm. Further, if $P_\sigma-\imath Q_\sigma$
depends holomorphically on $\sigma$ (for $\sigma$ in an open subset of
$\Cx$), then
$P_\sigma-\imath Q_\sigma$ is a holomorphic Fredholm family, while
$P_\sigma^*+\imath Q_\sigma^*$ is
antiholomorphic.
Note also that if $P_\sigma-\imath Q_\sigma$ is invertible (or if simply
$u\in\cX^s$, $f\in\cY^{s-m+1}$, $P_\sigma u=f$), and $\WF(f)$ is
disjoint from $\Lambda_\pm$
then
$\WF(P_\sigma^{-1}f)$ is also disjoint from this Lagrangian. Further,
if $f$ is $\CI$, then $P_\sigma^{-1}f$ is also $\CI$.
For the adjoint, corresponding to
propagation in the opposite direction, we have
$\WF((P^*_\sigma)^{-1}f)\subset \Lambda_+\cup\Lambda_-$ when $f$
is $\CI$.

For the semiclassical problem, a natural assumption is {\em non-trapping},
i.e.\ all semiclassical bicharacteristics in $\Sigma_\pm$
apart from those in the radial sets,
in $\Sigma_{\semi,\pm}$ are required to tend to
$L_\mp\cup\{\pm q_{\semi,z}>0\}$ in the forward direction and
$L_\pm\cup\{\pm q_{\semi,z}>0\}$ in the
backward direction. Here $L_\pm$ is the image of $\Lambda_\pm$
in $S^*\tilde X$ under the quotient map, and one considers $S^*\tilde
X$ as the boundary of the radial compactification of the fibers of
$T^*\tilde X$. Under this assumption, one has non-trapping
semiclassical estimates (analogues of hyperbolic estimates, i.e.\
with a loss of $h$ relative to elliptic estimates). This in particular
proves that for small $h$ the operator is invertible (not just
Fredholm), and thus the non-semiclassical Fredholm family has a
meromorphic inverse with finite rank poles.

This completes the
analytic ingredients in the non-trapping setting,
proving Theorem~\ref{thm:main} and Corollary~\ref{cor:main}.

We refer to \cite[Section~2,
Definition~2.18]{Vasy-Dyatlov:Microlocal-Kerr} for {\em semiclassical mildly
trapping} assumptions. These roughly state that there is a compact
subset $K$ of $T^*X$ (the `trapped set'), a neighborhood $O$ of $K$ and
a convex function $F$ on $T^*X$ which is $\geq 2$ on $K$ and $1$
outside $O$, and if one adds a complex absorption
$\tilde Q_\sigma$ which
vanishes near $K$ but is elliptic outside $O$,
then $(P_{h,z}-\imath\tilde Q_{h,z})^{-1}$ satisfies polynomial
bounds, $Ch^{-\varkappa-1}$, in $\im z>-C_0$ (i.e.\ in a strip without
the semiclassical rescaling), and such that the bicharacteristics of
$p_{\semi,z}$ are non-trapped once one regards $O$ as non-trapped,
i.e.\ entering $O$ in finite time is regarded as good as entering $\{\pm q>0\}$ in
finite time. Due to the gluing construction of
\cite{Datchev-Vasy:Gluing-prop}, semiclassical mildly trapping can be
immediately be combined with the analysis developed for non-trapping
$P_\sigma$, see \cite[Theorem~2.19]{Vasy-Dyatlov:Microlocal-Kerr},
roughly by placing complex absorption near $K$ but inside $O$ to
obtain a non-trapping `exterior' model, which can be glued with the
`interior' model
$(P_{h,z}-\imath\tilde Q_{h,z})^{-1}$.

\bibliographystyle{plain}
\bibliography{sm}

\end{document}